\documentclass[leqno]{amsart}
\usepackage{amsmath}
\usepackage{amssymb}
\usepackage{amsthm}
\usepackage{enumerate}
\usepackage[mathscr]{eucal}
\usepackage{graphicx}
\theoremstyle{plain}
\newtheorem{theorem}{Theorem}[section]
\newtheorem{prop}[theorem]{Proposition}
\newtheorem{cor}{Corollary}[theorem]

\theoremstyle{definition}
\newtheorem{definition}{Definition}[section]
\newtheorem{remark}{Remark}[section]

\newtheorem{example}{Example}[theorem]
\newtheorem*{acknowledgement}{\textnormal{\textbf{Acknowledgements}}}
\usepackage[pagewise]{lineno}
\usepackage{color}

\begin{document}
\title[On extreme contractions between real Banach spaces ]{On extreme contractions between real Banach spaces}
\author[Debmalya Sain, Kallol Paul and Arpita Mal]{Debmalya Sain, Kallol Paul and Arpita Mal}

\newcommand{\acr}{\newline\indent}

\address[Sain]{Department of Mathematics\\ Indian Institute of Science\\ Bengaluru 560012\\ Karnataka \\India\\ }
\email{saindebmalya@gmail.com}

\address[Paul]{Department of Mathematics\\ Jadavpur University\\ Kolkata 700032\\ West Bengal\\ INDIA}
\email{kalloldada@gmail.com}

\address[Mal]{Department of Mathematics\\ Jadavpur University\\ Kolkata 700032\\ West Bengal\\ INDIA}
\email{arpitamalju@gmail.com}

\thanks{ The research of the first author is sponsored by Dr. D. S. Kothari Postdoctoral fellowship. The third author would like to thank UGC, Govt. of India for the financial support.} 


\begin{abstract}
We completely characterize extreme contractions between two-dimensional strictly convex and smooth real Banach spaces, perhaps for the very first time. In order to obtain the desired characterization, we introduce the notions of (weakly) compatible point pair (CPP) and $ \mu- $compatible point pair ($ \mu- $CPP) in the geometry of Banach spaces. As a concrete application of our abstract results, we describe all rank one extreme contractions in $ \mathbb{L}(\ell_4^2, \ell_4^2) $ and $ \mathbb{L}(\ell_4^2, \mathbb{H}) $, where $\mathbb{H}$ is any Hilbert space.  We also prove that there does not exist any rank one extreme contractions in $ \mathbb{L}(\mathbb{H}, \ell_{p}^2), $ whenever $ p $ is even and $\mathbb{H}$ is any Hilbert space. We further study extreme contractions between infinite-dimensional Banach spaces and obtain some analogous results. Finally, we characterize real Hilbert spaces among real Banach spaces in terms of CPP, that substantiates our motivation behind introducing these new geometric notions. 
\end{abstract}

\maketitle
\section{Introduction.}

\medskip
Characterization of extreme contractions between Banach spaces is a classical problem in the geometry of Banach spaces. While the characterization problem has a definite solution for various particular pairs of Banach spaces \cite{G,Ga,I,K,Ki,Sh}, to the best of our knowledge, the problem remains unsolved in its most general form, even in the two-dimensional case. This is somewhat surprising, since which norm one operators are extreme contractions,  should be, at least in principle, entirely governed by the geometries of the domain and the range spaces. Indeed, this is the motivation behind our quest for a geometric property in the framework of Banach spaces, that would completely characterize the class of extreme contractions between Banach spaces. Our present work is an initial step in that direction. We introduce two different geometric concepts, whose importance in studying extreme contractions between Banach spaces would be self-explanatory in due course of time. As a matter of fact, we completely characterize extreme contractions on a two-dimensional strictly convex and smooth real Banach space, in terms of the geometric concepts introduced by us. Without further ado, let us establish the relevant notations and terminologies to be used throughout the paper.\\

Let $\mathbb{X},~ \mathbb{Y}$ be Banach spaces. Throughout the paper, we assume the Banach spaces to be real and of dimension greater than 1. Let $B_{\mathbb{X}} = \{x \in \mathbb{X} \colon \|x\| \leq 1\}$ and
$S_{\mathbb{X}} = \{x \in \mathbb{X} \colon \|x\|=1\}$ be the unit ball and the unit sphere of $\mathbb{X}$, respectively. Let $\mathbb{L}(\mathbb{X},\mathbb{Y})(\mathbb{K}(\mathbb{X},\mathbb{Y}))$ denote the Banach space of all bounded (compact) linear operators from $\mathbb{X}$ to $\mathbb{Y}$, endowed with the usual operator norm. We write $\mathbb{L}(\mathbb{\mathbb{X}, \mathbb{Y}})= \mathbb{L}(\mathbb{X})(\mathbb{K}(\mathbb{\mathbb{X}, \mathbb{Y}})= \mathbb{K}(\mathbb{X})), $ if $\mathbb{X}= \mathbb{Y}$. The notion of orthogonality in a Banach space plays an important role in determining the geometry of the underlying space. There are several notions of orthogonality in a Banach space, all of which are equivalent if the norm is induced by an inner product, i.e., if the Banach space is a  Hilbert space. For our present purpose, we require the following two notions of orthogonality in a Banach space:

For any two elements  $x,y \in {\mathbb{X}}$, $x$ is said to be Birkhoff-James \cite{B,J} orthogonal to $y$, written as $x \perp_B y$, if $ \|x+\lambda y\|\geq\|x\|$ $\forall \lambda \in \mathbb{R}$ and  $x$ is said to be isosceles orthogonal \cite{AMW} to $y$, written as $x \perp_I y$, if $\|x+y\|=\|x-y\|$.  For $x \in \mathbb{X}$, we write $x^{\perp}=\{y \in \mathbb{X}: x\perp_B y\}$.

As we will see in the later part of the paper, the norm attainment set of a bounded linear operator plays a vital role in the study of extreme contractions between Banach spaces. For a bounded linear operator $T$ defined on a Banach space $\mathbb{X}$, let $M_T$ denote the collection of all unit vectors in $\mathbb{X}$ at which $T$ attains norm, i.e,
\[ M_T= \{  x \in S_ \mathbb{X} \colon\|Tx\| = \|T\| \}. \]

As mentioned in the beginning, let us introduce the following two geometric notions in order to have a better understanding of extreme contractions between Banach spaces.\\

\begin{definition}
Let $ \mathbb{X},~\mathbb{Y} $ be two Banach spaces and $x \in S_{\mathbb{X}},~y \in S_{\mathbb{Y}}. $ We say that $(x,~y)$ is a weakly compatible point pair if there exists $z \in x^{\perp}\cap S_{\mathbb{X}}$,~ $w \in y^{\perp}\cap S_{\mathbb{Y}}$ and two constants $r>0,~\mu >0$ such that
\begin{eqnarray}
 ax+bz \in B(x,~ r)\cap S_{\mathbb{X}} \Longrightarrow \|ay+b \mu w\|\leq 1.
 \end{eqnarray}
In this case we write $(x,~y)$ is a weak CPP. Any pair of constants $ (r,~\mu), $ that works in the above sense, is said to be a weak CPP constant for the pair $ (x,~y). $\\
In particular, if $z_0 \in x^{\perp}\cap S_{\mathbb{X}}$,~ $w_0 \in y^{\perp}\cap S_{\mathbb{Y}}$ satisfy $(1)$ then we say that $(x,~y)$ is $\mu-$CPP with respect to $(z_0,~w_0)$.
\end{definition}

\begin{definition}
Let $ \mathbb{X},~\mathbb{Y} $ be two Banach spaces and $x \in S_{\mathbb{X}},~y \in S_{\mathbb{Y}}. $ We say that $(x,~y)$ is a compatible point pair if there exists two constants $r>0,~\mu >0$ such that for any $z \in x^{\perp}\cap S_{\mathbb{X}}$ and for any $w \in y^{\perp}\cap S_{\mathbb{Y}},~ $  $ax+bz \in B(x,~ r)\cap S_{\mathbb{X}} \Longrightarrow \|ay+b \mu w\|\leq 1.$  
In this case we write $(x,~y)$ is a CPP. Any pair of constants $ (r,~\mu), $ that works in the above sense, is said to be a CPP constant for the pair $ (x,~y). $
\end{definition}

In light of these newly introduced geometric notions, we study extreme contractions between Banach spaces. First we explore some basic properties of these notions. We further study the connection between operator norm attainment and weak CPP. We also prove that in a finite-dimensional smooth Banach space $ \mathbb{X}, $ there exists nontrivial weak CPP. Next, we establish an interesting connection between weak CPP and a norm one linear operator of rank one being an extreme contraction. We extend the result for rank one bounded linear operators from a smooth and reflexive Banach space with Kadets-Klee property to a smooth Banach space. Let us recall that a normed space $\mathbb{X}$ is said to have Kadets-Klee property if for any sequence $\{x_n\}$ in $\mathbb{X}$, $x_n \rightharpoonup x$ and $\|x_n\| \longrightarrow \|x\|$ implies that $x_n \longrightarrow x$. As we will see, using the notion of CPP, it is possible to completely characterize rank one extreme contractions between two-dimensional strictly convex and smooth Banach spaces. As a concrete application of our result, we describe all rank one extreme contractions in $ \mathbb{L}(\ell_4^2,\ell_4^2) $ and $ \mathbb{L}(\ell_4^2,\mathbb{H}) $, where $\mathbb{H}$ is any Hilbert space. We also prove that there does not exist any rank one extreme contractions in $\mathbb{L}(\mathbb{H}, \ell_{p}^2), $ whenever $ p $ is even and $\mathbb{H}$ is any Hilbert space. However, in order to characterize rank two extreme contractions between two-dimensional strictly convex and smooth Banach spaces, we do require the notion of $ \mu- $CPP. Combining our results, we obtain a complete characterization of extreme contractions between two-dimensional strictly convex and smooth Banach spaces. Indeed, this result can be viewed as an abstract generalization of the characterization of extreme contractions between  $ \ell_p^2 $ spaces by Grzaslewicz \cite{G}, Bandyopadhyay and Roy \cite{BR} and is a major highlight of the present paper. For a particular class of compact operators from a smooth and reflexive Banach space $ \mathbb{X} $ with Kadets-Klee property to a smooth Banach space $\mathbb{Y}$, we prove an analogous connection between extreme contractions and CPP. As some further applications of extreme contractions and CPP in the study of geometry of Banach spaces, we relate these notions with strict convexity. We also obtain a characterization of real Hilbert space in terms of CPP, namely, a real Banach space $\mathbb{X}$ is a Hilbert space if and only if $(x,~x)$ is a CPP for each $x \in \mathbb{X}$ with CPP constant $(r,~1)$ for any $r>0$. We would like to remark that apart from distinguishing Hilbert spaces among Banach spaces, this further substantiates our  geometric motivation for introducing the concept of CPP. 

\section{Main Results}  
Let us begin with a simple proposition which will be used throughout the paper. 
\begin{prop}\label{prop:convexity}
	Let $\mathbb{X}$ be a normed linear space,  $u,~v \in \mathbb{X}$ and   $0< \lambda_0 \in \mathbb{R}.$ If $\|u\|\leq 1$ and $\|u+\lambda_0 v\|\leq 1$ then $\|u+\lambda v\|\leq 1$ for all $0 \leq \lambda \leq \lambda_0$.
\end{prop}
\begin{proof}
	Follows easily from the convexity of the norm.
\end{proof}

In the next proposition we discuss some of the basic properties related to the notions of (weak) CPP.
\begin{prop}\label{proposition:properties}
	Let $ \mathbb{X},~\mathbb{Y} $ be Banach spaces and $x \in S_{\mathbb{X}},~y\in S_{\mathbb{Y}}.$ Then the following are true:\\
(i) $(x,~y)$ is a CPP implies that $(x,~y)$ is a weak CPP, but the converse may not be true.\\
(ii) In a Hilbert space $\mathbb{H}$, $(x,~y)$ is a CPP for any $x,~y \in S_{\mathbb{H}}.$ In this case the CPP constant may be chosen as $ (r,~\mu), $ with any $ r>0 $ and $ \mu =1. $\\
(iii) $(x,~ x)$ is always a weak CPP but $(x,~x)$ may not be a CPP.\\
(iv) $(x,~y)$ may or may not be a weak CPP.\\
(v) If $(r_{0},~\mu_{0})$ is a (weak) CPP constant for the pair $ (x,~y), $ then $ (r,~\mu) $ is a (weak) CPP constant for the pair $ (x,~y), $ for any $ 0 < r \leq r_{0} $ and for any $ 0 < \mu \leq \mu_{0}. $ \\
(vi) If $(x,~y)$ is a CPP and $x$ is not an extreme point of $B_{\mathbb{X}}$ then $y$ is not an extreme point of $B_{\mathbb{Y}}$. But the converse is not true. 
\end{prop}
\begin{proof}
	$(i)$ First part follows trivially from definition. On the other hand, notice that, 
	in $\ell_{\infty}^{2}$, $((1,1),~(1,-1))$ is a weak CPP but not CPP.\\
	$(ii)$ Let $x,~y \in S_{\mathbb{H}}$. Let $z \in x^{\perp}\cap S_{\mathbb{H}}$ and $w \in y^{\perp}\cap S_{\mathbb{H}}$. Then $\|ax+bz\|=1 \Longrightarrow a^2+b^2=1$. Now, $\|ay+bw\|^2=a^2+b^2=1 \Longrightarrow \|ay+bw\|=1$. Hence, $(x,~y)$ is a CPP.\\
	$(iii)$ From the definition of weak CPP, it is clear that for any $x\in S_{\mathbb{X}}$, $(x,~x)$ is always a weak CPP. But in  $\ell_{\infty}^{2}$, $((1,1),~(1,1))$ is not a CPP.\\
	$(iv) $ Simply observe that in $\ell_{\infty}^{2}$, $((1,0),~(1,1/2)) $ is a weak CPP but $((1,0),~(1,1))$ is not a weak CPP.\\
	$ (v)  $ We give the proof for weak CPP, the proof for CPP follows similarly. Since $(x,~y)$ is a weak CPP with CPP constant $(r_0,~ \mu_0)$ and $0<r \leq r_0$, there exists  $z\in x^{\perp}\cap S_{\mathbb{X}}$ and $w \in y^{\perp}\cap S_{\mathbb{Y}}$, $ax+bz\in B(x,~r)\cap S_{\mathbb{X}}\Longrightarrow ax+bz\in B(x,~r_0)\cap S_{\mathbb{X}}\Longrightarrow \|ay+b \mu_0 w\|\leq 1 $. Since  $0< \mu \leq \mu_0$, using Proposition \ref{prop:convexity}, we have, $\|ay+b\mu w\|\leq 1$. Thus $(x,~y)$ is a weak CPP with weak CPP constant $(r,~\mu)$. \\
	$(vi) $ Let $(x,~y)$ be a CPP with CPP constant $(r,~\mu)$. Since $x$ is not an extreme point of $B_{\mathbb{X}}$, there exists $x_1,~x_2 \in B_{\mathbb{X}}$ such that $x\neq x_1,~x\neq x_2$ and $x=\frac{1}{2}(x_1+x_2)$. It is easy to check that $x \perp_B (x_1-x)$. Clearly, there exists $0<t <1$ such that $x+t(x_1-x)=(1-t)x+tx_1 \in B(x,~r)\cap S_{\mathbb{X}}$ and $x-t(x_1-x)=(1-t)x+tx_2 \in B(x,~r)\cap S_{\mathbb{X}}$. Since $(x,~y)$ is a CPP, for any $w \in y^{\perp}\cap S_{\mathbb{Y}}$, we have, $\|y\pm t \mu w\|\leq 1$. Also $y \perp_B w$ implies that $\|y \pm t \mu w\|\geq 1$ and so $\|y\pm t\mu w\|=1$. Clearly $y=\frac{1}{2}(y+t \mu w)+\frac{1}{2}(y-t \mu w)$. Hence $y$ is not an extreme point of $B_{\mathbb{Y}}$.\\
	But in $\ell_{\infty}^2,~((1,1),~(1,0))$ is a CPP where $(1,0)$ is not an extreme point of $B_{\ell_{\infty}^2}$ and $(1,1)$ is an extreme point of $B_{\ell_{\infty}^2}$.
\end{proof}

In the next theorem, we establish a connection between operator norm attainment and weak CPP, for bounded linear operators between smooth Banach spaces. 
 
\begin{theorem}\label{theorem:weak}
	Let $\mathbb{X},~\mathbb{Y}$ be smooth Banach spaces. Let $T \in \mathbb{L}(\mathbb{X}, \mathbb{Y})$ be such that  $rank~T>1$ and $\|T\|=1$. Let $M_T \neq \emptyset$ and $x \in M_T$. Then  $(x,~ Tx)$ is a weak CPP.
\end{theorem}
\begin{proof}
	If $Tz= 0$ for all $z \in x^{\perp} \cap S_{\mathbb{X}}$ then $rank~T = 1$, a contradiction to our hypothesis that $rank~T>1$. Therefore,  there exists $z \in x^{\perp} \cap S_{\mathbb{X}}$ such that $Tz \neq 0$. Since $\mathbb{X},~\mathbb{Y}$ are smooth, $x \in M_T$ and $x \perp_B z$, from Theorem 2.2 of \cite{S},  we have, $Tx \perp_B Tz$. If possible, suppose that $(x,~Tx)$ is not a weak CPP. Then, for any $r >0$, $0< \mu < \|Tz\|$, there exists $a,~ b \in \mathbb{R}$ such that $a x + b z \in B(x,~r) \cap S_{\mathbb{X}}$ but $\|a Tx + b \mu \frac{Tz}{\|Tz\|}\|>1. $ In other words, $ \|Tw\|>1 $, where $w= a x+ \frac{b \mu}{\|Tz\|}z $. Since $ x \perp_{B} z, $ it follows that $|a|\leq \|a x + b z\|=1$. Since $0< \frac{\mu}{\|Tz\|}<1$, using Proposition \ref{prop:convexity}, we have $\|w\|\leq 1$. Now, $\|Tw\|> 1$ and $\|w\|\leq 1$ together imply that $\|T\|> 1$, a contradiction to our assumption that $\|T\|=1$. Therefore, $(x,~Tx)$ is a weak CPP. This completes the proof of the theorem.
\end{proof}
\begin{cor}\label{cor:invertible}
	Let $\mathbb{X}$ be a reflexive, smooth Banach space and $\mathbb{Y}$ be a smooth Banach space. Let $T \in \mathbb{K}(\mathbb{X},\mathbb{Y})$ be a norm one linear operator such that $T$ is invertible. Then $M_T \neq \emptyset$. Moreover, for each $x \in M_T, $  $(x,~Tx)$ is a weak CPP.
\end{cor}	
\begin{proof}
	Since every compact operator on a reflexive Banach space attains norm, it follows that $M_T \neq \emptyset$. Since $T$ is invertible, it follows from Theorem \ref{theorem:weak}, that $(x,~Tx)$ is a weak CPP for each $x \in M_T. $ 
\end{proof}

We have already observed in Proposition \ref{proposition:properties} that given any $ x \in S_{\mathbb{X}}, $ $ (x,~ x) $ is a weak CPP. In the next Theorem, we establish the existence of nontrivial weakly compatible point pair(s) in any finite-dimensional smooth Banach space.

\begin{theorem}
	Let $\mathbb{X}$ be a finite-dimensional smooth Banach space. Then there exists $x,~y \in S_{\mathbb{X}}$, such that $y\neq \pm x$ and $(x,~y)$ is a weak CPP.
\end{theorem}
\begin{proof}
	Let us first suppose that $\mathbb{X}$ is two-dimensional. Let $\{x_1,x_2\}$ be a basis of $\mathbb{X}$. Define $S\in \mathbb{L}(\mathbb{X})$ by $Sx_1= x_2$ and $Sx_2= - x_1$. Clearly, $S$ has no real eigenvalue. Let $T= \frac{1}{\|S\|}S$. Clearly, $T$ is invertible and $T$ has no real eigenvalue. Let $x\in M_T$. Then by Corollary \ref{cor:invertible}, we have, $(x,~Tx)$ is a weak CPP. Clearly, $Tx \neq \pm x$, since otherwise, $\pm 1$ will be an eigenvalue  of $T$, a contradiction to the fact that $T$ has no real eigenvalue. This completes the proof of the theorem in the special case of $ \mathbb{X} $ being two-dimensional.\\
	Now, let $\mathbb{X}$ be n-dimensional ($n \geq 3$). By Theorem 4.33 of \cite{AMW}, there exist two linearly independent vectors $x_1,~x_2$ in $\mathbb{X}$ such that $\|x_1+ x_2\|=\|x_1-x_2\|$, i.e., $x_1,~ x_2$ are isosceles orthogonal. Let $\{x_1,~x_2,~x_3, \ldots, x_n\}$ be a basis of $\mathbb{X}$. Define $S\in \mathbb{L}(\mathbb{X})$ as follows:
	\begin{eqnarray*}
	Sx_1&=& x_2,\\
	Sx_2&=&-x_1,\\
	Sx_i&=&\frac{1}{2}x_i, ~ \forall ~3 \leq i \leq n.
	\end{eqnarray*}
	Clearly, $S$ is invertible and $\frac{1}{2}$ is the only real eigenvalue of $S$. Now, $S(x_1+x_2)=x_2-x_1 \Longrightarrow \|S(\frac{x_1+x_2}{\|x_1+x_2\|})\|=\frac{\|x_2-x_1\|}{\|x_1+x_2\|}=1 \Longrightarrow \|S\|\geq 1$. Let $T= \frac{1}{\|S\|}S$. Then $\frac{1}{2\|S\|}$ is the only real eigenvalue of $T$ and $0< \frac{1}{2\|S\|} \leq \frac{1}{2}<1$. Let $x\in M_T$. Applying Corollary \ref{cor:invertible}, we have, $(x,~ Tx)$ is a weak CPP. Clearly, $Tx \neq \pm x$, since otherwise, $\pm 1$ will be an eigenvalue  of $T$, a contradiction. This proves the theorem.
\end{proof}

In a Hilbert space $ \mathbb{H}, $  a bounded linear operator in $ \mathbb{L}(\mathbb{H}) $ is an extreme contraction if and only if it is an isometry or a coisometry \cite{Ga,K}. In particular, there does not exist a rank one bounded linear operator in $ \mathbb{L}(\mathbb{H}) $ which is an extreme contraction. However, the scenario is dramatically different in case of bounded linear operators between general Banach spaces, which are not necessarily Hilbert spaces. Indeed, it follows from the works of Grzaslewich \cite{G} that there exists rank one extreme contractions in $ \mathbb{L}(\ell_{p}^{2}). $ In the next theorem, we give a sufficient condition for a rank one bounded linear operator from a smooth Banach space to a strictly convex and smooth Banach space to be an extreme contraction.

\begin{theorem}
	Let $\mathbb{X}$ be a smooth Banach space and $\mathbb{Y}$ be a strictly convex and smooth Banach space. Let $T\in \mathbb{L}(\mathbb{X}, \mathbb{Y})$ be such that $rank~T=1$ and $\|T\|=1$. If $(x, ~Tx)$ is not a weak CPP for some $x \in M_T,$ then $T$ must be an extreme contraction.
\end{theorem}
\begin{proof}
	Since $x\in M_T$ and $rank~T=1$, applying Theorem $ 2.2 $ of \cite{S}, we have, $Ty=0$ for all $y \in x^{\perp} \cap S_{\mathbb{X}}$. If possible, suppose that $T$ is not an extreme contraction. Then there exists $T_1,~ T_2 \in \mathbb{L}(\mathbb{X}, \mathbb{Y})$ such that $T= \frac{1}{2}T_1+ \frac{1}{2}T_2$, where $T_1 \neq T,~ T_2 \neq T,~ \|T_1\|\leq 1$ and $\|T_2\|\leq 1$. Now, $Tx= \frac{1}{2}T_1 x+ \frac{1}{2}T_2 x$ implies that $T_1 x= T_2 x= Tx$, since $\mathbb{Y}$ is strictly convex. Thus, $x\in M_{T_1}\cap M_{T_2}$. Clearly, there exists $y\in x^{\perp}\cap S_{\mathbb{X}}$ such that $T_1 y\neq 0$, otherwise, $T_1$ will be equal to $T$, a contradiction. Since $ \mathbb{X},~\mathbb{Y} $ are smooth, $x \in M_{T_1}$ and  $y\in x^{\perp}\cap S_{\mathbb{X}}$, applying Theorem 2.2 of \cite{S}, we have, $T_1 x \perp_B T_1 y$, i.e., $Tx \perp_B T_1 y$. Therefore, $T_1 y= \mu w$ for some $w\in (Tx)^{\perp}\cap S_{\mathbb{Y}}$ and for some nonzero real number $\mu$.  Now, $Ty= \frac{1}{2}T_1 y+ \frac{1}{2}T_2 y$ implies that $T_2 y= -\mu w$. Without loss of generality assume that $\mu >0$. Now, since $(x,~ Tx)$ is not a weak CPP, for $r >0,~ \mu>0$ there exists $a,~b \in \mathbb{R}$ such that $ax+ by \in B(x,~r)\cap S_{\mathbb{X}}$ but $\|aTx+b \mu w\|>1. $ Therefore, we have, $ \|aT_1 x+ b T_1 y\|>1 \Longrightarrow \|T_1(ax+by)\|>1.$ Since $ \| ax+by \| =1, $ we must have, $ \| T_1 \| > 1, $ a contradiction to our assumption that $\|T_1\|\leq 1$. Therefore, $T$ must be an extreme contraction. This completes the proof of the theorem. 
\end{proof}

In case, $\mathbb{X}$ is a two-dimensional smooth Banach space and  $ \mathbb{Y}$ is a two-dimensional strictly convex and smooth Banach space,  the above theorem can be substantially strengthened to obtain the following theorem:

\begin{theorem}\label{theorem:2dimrank1}
Let $\mathbb{X},~\mathbb{Y}$ be two-dimensional smooth Banach spaces and in addition, $\mathbb{Y}$ be  strictly convex. Let $ T$ be a rank one operator from $\mathbb{X}$ to $\mathbb{Y}$ with $\|T\|=1$. If  $(x,~ Tx)$ is not a CPP for some $x \in M_T$, then $T$ is an extreme contraction.
\end{theorem}
\begin{proof}
 If possible, suppose that $T$ is not an extreme contraction. Then there exists $T_1,~T_2 \in \mathbb{L}(\mathbb{X}, \mathbb{Y})$ such that $T=\frac{1}{2}(T_1+T_2)$, $T \neq T_1,~T\neq T_2,~\|T_1\|=\|T_2\|=1$. Clearly, $Tx=  \frac{1}{2}(T_1 x+T_2 x)$. Since $Tx$ is an extreme point of $ B_{\mathbb{Y}} $, we have, $Tx=T_1 x= T_2 x$, i.e., $x \in M_{T_1} \cap M_{T_2}$. Therefore, applying Theorem $ 2.2 $ of \cite{S}, we have, $x\perp_B y \Longrightarrow T_1 x \perp_B T_1 y \Longrightarrow Tx \perp_B T_1y$. Now, since $rank~T=1$ and $x \perp_B y, $ we must have, $ Ty=0. $ Since $ T_1 \neq T, $ it is easy to deduce that $T_1 y \neq 0$. Therefore, $T_1 y = \mu w, ~T_2 y=- \mu w$ for some $\mu \in \mathbb{R} \setminus \{0\}$ and $w \in (Tx)^{\perp} \cap S_{\mathbb{Y}}$. Without any loss of generality, let us assume that $\mu >0$. Let us also observe that since $ \mathbb{Y} $ is smooth, if $ u \in S_{\mathbb{X}} $ is such that $ Tx \perp_{B} u $ then $ u = \pm w. $ Now, since $(x,~Tx)$ is not a CPP, for $r>0,~\mu >0,$ there exists $a, ~b \in \mathbb{R}$ such that $ax+ by \in B(x,~r) \cap S_{\mathbb{X}}$ but either $\|aTx+ b \mu w\|>1$ or $\|aTx-b \mu w\|>1.$ In the first case, $\|T_1(ax+by)\|>1$ and in the second case, $\|T_2(ax+by)\|>1$. Since $ \| ax+by \| = 1, $ we have, either $ \| T_1 \| > 1 $ or $ \| T_2 \| > 1. $ However, this contradicts our initial assumption that $\|T_1\|=\|T_2\|=1$.  Therefore, $T$ must be an extreme contraction. This completes the proof of the theorem.
\end{proof}

The restriction on the dimension of the domain space $\mathbb{X}$ (codomain space $ \mathbb{Y})$ can be relaxed if the corresponding norm on $\mathbb{X}$ (codomain space $ \mathbb{Y}) $ is induced by an inner product. We observe that in a Hilbert space $\mathbb{H},$ if $x,~y,~u,~v \in S_{\mathbb{H}}$ are such that $x \perp y$ and $u \perp v$ then $\|ax+by\|^2=\|au+bv\|^2 = a^2 + b^2$. We now state the following theorem whose proofs follow in the same line of arguments of Theorem \ref{theorem:2dimrank1}.

\begin{theorem}\label{theorem:hilbertrank1}
	Let $\mathbb{H}$ be any Hilbert space and $\mathbb{X},\mathbb{Y}$ be  two-dimensional strictly convex smooth Banach spaces. Let $ T$ be a rank one operator from $\mathbb{H}$ to $\mathbb{Y}$ (from $\mathbb{X}$ to $\mathbb{H}$) with $\|T\|=1$.  If  $(x,~ Tx)$ is not a CPP for $x \in M_T$, then $T$ is an extreme contraction.
\end{theorem}

In the following theorem, in a smooth and  reflexive Banach space with Kadets-Klee property, we give a sufficient condition for a rank one bounded linear operator $ T $ to be such that $ T $ is not an extreme contraction. 

\begin{theorem}\label{theorem:rank1}
  Let $\mathbb{X}$ be a smooth reflexive Banach space with Kadets-Klee property and $\mathbb{Y}$ be a smooth Banach space. Let $ T$ be a rank one operator from $\mathbb{X}$ to $\mathbb{Y}$ with $\|T\|=1$.  If $M_T=\{\pm x\}$  and $(x,~ Tx)$ is a CPP, then $T$ is not an extreme contraction.
\end{theorem}
\begin{proof}
Since $ \mathbb{X} $ is smooth, $rank ~T=1$ and $x \in M_T$, once again applying Theorem $ 2.2 $ of \cite{S}, it is easy to check that $Th=0$ for all $h \in x^{\perp}$. We also note that since $\mathbb{X}$ is smooth, there exists a unique hyperspace $H_1$ of codimension $1$ such that $x \perp_B H_1$. Choose $y \in H_1 \cap S_{\mathbb{X}}$. Then there exists a unique hyperspace $H_2$ in $H_1$ such that $y \perp_B H_2$. Hence, every $z \in \mathbb{X}$ can be uniquely written as $ax+ h_1$, where $a \in \mathbb{R},~h_1 \in H_1$. Again, $h_1$ can be uniquely written as $by+h_2$, where $b\in \mathbb{R},~h_2 \in H_2$. Therefore, every $z \in \mathbb{X}$ can be uniquely written as $ax+by+ h_2$, where $a,~b \in \mathbb{R},~h_2 \in H_2$.
	 Choose $w \in S_{\mathbb{Y}}$ such that $Tx \perp_B w$. For each $n \in \mathbb{N},$ define $T_n,~ S_n$ in the following way:
	\[T_n(z)=a Tx+\frac{b}{n}w,~ S_n(z)=a Tx-\frac{b}{n}w.  \]
  It is easy to see that $T_n,~ S_n \in \mathbb{K}(\mathbb{X}, \mathbb{Y})$.  Clearly, $T= \frac{1}{2}T_n + \frac{1}{2}S_n$ and $T\neq T_n$, $T\neq S_n$  for all $n\in \mathbb{N}$. If possible, suppose that $T$ is an extreme contraction. Then either $\|T_n\|>1$ or $\|S_n\|>1$ for each $n \in \mathbb{N}$. Without loss of generality (passing on to a subsequence, if necessary), let us assume that $\|T_n\|>1$ for all $n \in \mathbb{N}$. Each $T_n$, being a compact operator on a reflexive Banach space, attains its norm. Let $y_n \in M_{T_n}$ for each $n \in \mathbb{N}$. \\
	Now, since $(x,~ Tx)$ is a CPP, there exists $r>0,~\mu >0$ such that $ax+bh \in B(x,~ r)\cap S_{\mathbb{X}} \Longrightarrow \|aTx + b \mu u\|\leq 1 $ for all $h \in H_1 \cap S_{\mathbb{X}}$ and for all $u \in (Tx)^{\perp} \cap S_{\mathbb{Y}}$.  Choose $n_0 \in \mathbb{N}$ such that $0< \frac{1}{n}< \mu$ for all $n \geq n_0$. We show that $\|T_n z\|\leq 1$ for all $n \geq n_0$. Let $z= (a x + b y + h_2) \in B(x,~ r)\cap S_{\mathbb{X}}$, where, $h_2 \in H_2$. If $by+h_2= 0$ then clearly $\|T_n z\| \leq 1$. Let $by+h_2 \neq 0$. Then $z=ax+ch$, where, $c=\|by+h_2\|$ and $h=\frac{1}{c}(by+h_2)$. Then $|b| \leq \|by+h_2\|=c$, since $y \perp_B h_2$. Hence, $z=ax+ch \in B(x,~ r) \cap S_{\mathbb{X}} \Longrightarrow \|a Tx \pm c \mu w\|\leq 1$. 
    Since $|b|\leq c,~  -c \mu< \frac{b}{n}< c \mu$ for all $n \geq n_0$. Therefore, using Proposition \ref{prop:convexity},  we have, for all $n \geq n_0$,
		\[\|T_n z\|=\|T_n(ax+by+h_2)\|=\|a Tx + \frac{b}{n} w\| \leq 1 .\]
Since $y_n \in M_{T_n} $ and $ \|T_n \| >1, $ we must have, $y_n \notin B(x,~ r)$. Similarly, it can be shown that  $y_n \notin B(-x,~ r)$. Since $\mathbb{X}$ is reflexive, $B_{\mathbb{X}}$ is weakly compact. Therefore, $\{y_n\}$ has a weakly convergent subsequence in $B_{\mathbb{X}}$. Without loss of generality, let us assume that $\{y_n\}$ weakly converges to $y_0 \in B_{\mathbb{X}}$.\\
    Now, let $z=(a x + b y + h_2)\in S_{\mathbb{X}}$. Then $|a|\leq \|(a x + by+h_2)\|=1$ and $|b|\leq \|( b y + h_2)\|= \|z-a x\|\leq 1+ |a| \leq 2$. Therefore, $\|(T_n-T)z\|=\|(T_n-T)(ax+by+h_2)\|=\|\frac{b}{n}w\|= \frac{|b|}{n}\leq \frac{2}{n}$. This implies that $\|(T_n-T)z\|\longrightarrow 0$ as $n \longrightarrow \infty$. Hence, $T_n \longrightarrow T$. Since $T$ is a compact operator and $\{y_n\}$ weakly converges to $y_0$, $\{Ty_n\}$ converges strongly to $Ty_0$. Therefore, $\|T_n y_n- Ty_0\|\leq \|T_n y_n-Ty_n\|+ \|Ty_n-Ty_0\|\leq \|T_n-T\|+\|T y_n-Ty_0\|\longrightarrow 0$  as $n \longrightarrow \infty$. Hence, $T_ny_n \longrightarrow Ty_0 \Longrightarrow \|T_n y_n\|\longrightarrow \|Ty_0\| \Longrightarrow \|Ty_0\|\geq 1$, since $\|T_n y_n\|>1$ for all $n \in \mathbb{N}$. Since $\|T\|=1 $ and $y_0 \in B_{\mathbb{X}}, $  $ \|Ty_0\|=1$. Therefore, $y_0 = \pm x $. Again, $y_n \rightharpoonup y_0$ and $\|y_n\|\longrightarrow \|y_0\|$. Hence, $y_n \longrightarrow y_0$. This implies that $y_0 \notin B(\pm x,~ r)$, a contradiction.  Therefore, $T$ is not an extreme contraction. This establishes the theorem.
\end{proof}
We would like to note that combining Theorem \ref{theorem:2dimrank1} and Theorem \ref{theorem:rank1}, we obtain a complete characterization of rank one extreme contractions from a two-dimensional smooth Banach space to a two-dimensional strictly convex and smooth Banach space. We state the characterization in form of the following theorem:

\begin{theorem}\label{theorem:completerank1}
	Let $\mathbb{X},~\mathbb{Y}$ be two-dimensional smooth Banach spaces and in addition, $\mathbb{Y}$ be strictly convex. Let $ T$ be a rank one operator from $\mathbb{X}$ to $\mathbb{Y}$ with $\|T\|=1$. Let $M_T=\{\pm x \}$.  Then $T$ is an extreme contraction if and only if $(x,~ Tx)$ is not a CPP .
\end{theorem}

\begin{remark}
In addition, if we assume that $\mathbb{X}$ is strictly convex then $rank~T=1$ will imply that $M_T=\{\pm x\}$ and so the assumption that $M_T=\{\pm x\}$ can be relaxed from the hypothesis of Theorem \ref{theorem:rank1} and Theorem \ref{theorem:completerank1}.  
\end{remark}

As before, if the domain space $\mathbb{X}$ (codomain space $\mathbb{Y})$ is a Hilbert space then combining Theorem \ref{theorem:hilbertrank1} and Theorem \ref{theorem:rank1}, we have the following two characterizations:

\begin{theorem}\label{theorem:completeHrank1}
	Let $\mathbb{H}$ be a Hilbert space and $\mathbb{Y}$ be a  two-dimensional smooth and strictly convex Banach space. Let $ T$ be a rank one operator from $\mathbb{H}$ to $\mathbb{Y}$ with $\|T\|=1$. Then $M_T=\{\pm x \}$. Moreover, $T$ is an extreme contraction if and only if $(x,~ Tx)$ is not a CPP .
\end{theorem}

\begin{theorem}\label{theorem:completehrank1}
	Let $\mathbb{X}$ be a  two-dimensional smooth Banach space and $\mathbb{H}$ be a Hilbert space. Let $ T$ be a rank one operator from $\mathbb{X}$ to $\mathbb{H}$ with $\|T\|=1$. Let $M_T=\{\pm x \}$. Then $T$ is an extreme contraction if and only if $(x,~ Tx)$ is not a CPP .
\end{theorem}

As concrete applications of Theorem \ref{theorem:completerank1}, we would completely describe rank one extreme contractions between two-dimensional strictly convex smooth Banach spaces, in three special cases. First, we completely describe rank one extreme contractions in $ \mathbb{L}(\ell_{4}^{2}). $ In view of Theorem \ref{theorem:completerank1}, our strategy is to determine when a pair of points on the unit sphere of $ \ell_{4}^{2} $ is not a CPP. To this end, we have the following theorem: 

\begin{theorem}\label{theorem:l4}
	Let $ \mathbb{X} = \ell_4^2. $ Then $((x,~y),~(x_1,~y_1)) \in S_{\mathbb{X}} \sf X S_{\mathbb{X}} $ is not a CPP if and only if $xy=0$ and $x_1y_1 \neq 0$.
\end{theorem}	
\begin{proof}  	First we prove the sufficient part of the theorem. Suppose that  $xy=0$ and $x_1y_1 \neq 0$. If possible, suppose that $((x,~y),~(x_1,~y_1))$ is a CPP with CPP constant $(r,~ \mu)$. We observe that in $ \ell_4^2, (x,y) \bot_B \lambda (y^3,-x^3) $ where $ \lambda \in \mathbb{R}.$  
	Let $k=\frac{1}{\|(y^3,-x^3)\|}$ and $k_1=\frac{1}{\|(y_1^3,-x_1^3)\|}$.  
	   Since $xy=0$, $(x,~y)=(\pm 1,0)$ or $(0,\pm 1)$. Let $(x,~y)=( 1,0)$. Then $(a,b)\in B(( 1,0),~r)\cap S_{\mathbb{X}} \Rightarrow \|a(x_1,~y_1)+b\mu k_1(y_1^3,-x_1^3)\|^4 \leq 1$ and so   
	\begin{eqnarray*}
	a^4+6a^2b^2 k_1^2 x_1^2 y_1^2{\mu}^2+4ab^3k_1^3x_1y_1(y_1^4-x_1^4){\mu}^3+b^4{\mu}^4 &\leq & a^4+b^4\\ 
	\Rightarrow 6a^2 k_1^2 x_1^2 y_1^2{\mu}^2+4abk_1^3x_1y_1(y_1^4-x_1^4){\mu}^3+b^2({\mu}^4-1) &\leq & 0 \ldots (1) 
	\end{eqnarray*} 
	 We can find sufficiently small $b$ for which $(a,b)\in B(( 1,0),~r)\cap S_{\mathbb{X}} $ but (1) does not hold. Therefore, $((1,0),~(x_1,~y_1))$ is not a CPP. Similarly, it can be shown that  $((-1,~0),~(x_1,~y_1))$ is not a CPP and $((0, \pm 1),~(x_1,~y_1))$  is not a CPP. \\
	Next, we prove the necessary part of the theorem.	Suppose that  $((x,~y),~(x_1,~y_1)) $ is not a CPP.\\
	If possible, suppose that  $xy=x_1y_1= 0$. Then it is easy to verify that $((x,~y),~(x_1,~y_1))$ is a CPP, contradicting our assumption. Therefore, either $xy\neq 0$ or $x_1y_1 \neq 0$.\\
	Next, let $xy \neq 0$. Choosing $0<r< min \{\frac{3|xy|}{8k+6|xy|},~1\}$, we can find a $\mu > 0$ such that  $a(x,y)+bk(\mp y^3,\pm x^3)\in B((x,y),~r)\cap S_{\mathbb{X}}$ implies that $\|a(x_1,y_1) + b \mu k_1(\mp y_1^3, \pm x_1^3)\| \leq 1. $
	This shows that $((x,~y),~(x_1,~y_1)) $ is  a CPP if $ xy \neq 0.$ We note that the case  $ x_1y_1 \neq 0 $ has already being taken care of.  Hence, if $((x,~y),~(x_1,~y_1))$ is not a CPP we must have $xy=0$ and $x_1y_1 \neq 0$.
	This completes the proof of the theorem.	   
\end{proof}

Now we completely describe rank one extreme contractions in $ \mathbb{L}(\ell_{4}^{2}), $ by means of the following theorem:

\begin{theorem}\label{theorem:ellrank1}
	In $ \mathbb{L}(\ell_4^2), $ the matrix representations (with respect to the standard ordered basis) of the class of all rank one  extreme contractions are given by $$
	\begin{bmatrix}
	x_1 & 0\\
	y_1 & 0	
	\end{bmatrix}
	,
	\quad
	\begin{bmatrix}
	0 & x_1\\
	0 & y_1	
	\end{bmatrix}
	,$$ 
	where, $x_1y_1 \neq 0$ and $x_1^4+y_1^4=1$.
\end{theorem}	
\begin{proof}
	Let $T$ be a rank one operator on $\ell_4^2$ such that $\|T\|=1$. Suppose $T$ attains norm at $(x,y)$ and $T(x,y)=(x_1,y_1)$. Then from Theorem \ref{theorem:completerank1}, we have, $T$ is an extreme contraction if and only if $((x,y),(x_1,y_1))$ is not a CPP. Therefore, by Theorem \ref{theorem:l4}, we have, $xy=0$ and $x_1y_1 \neq 0$. Now, $ xy = 0 $ implies that either $ (x,y) = \pm (1,0) $ or $ (x,y) = \pm (0,1). $ Since $ T $ is rank one, applying Theorem $ 2.2 $ of \cite{S}, it is easy to deduce that in the first case, $ T(0,1)=(0,0), $ whereas in the second case, $ T(1,0)=(0,0). $  Therefore, the matrix representation of $T$ with respect to the standard ordered basis is necessarily of the form  $$
	\begin{bmatrix}
	x_1 & 0\\
	y_1 & 0	
	\end{bmatrix}
	,
	\quad
	\begin{bmatrix}
	0 & x_1\\
	0 & y_1	
	\end{bmatrix}
	,$$ 
	where, $x_1y_1 \neq 0$ and $x_1^4+y_1^4=1$.
\end{proof}

We would like to remark that Theorem \ref{theorem:ellrank1} falls completely within the scope of the analysis of extreme contractions in $ \mathbb{L}(\ell_{p}^{2}), $ presented in \cite{G}. However, the true strength of Theorem \ref{theorem:completerank1} and Theorem \ref{theorem:completeHrank1} lies in the fact that it is applicable  for general two-dimensional spaces and furthermore, we are allowed to take the domain space and the range space different from one another. In the next theorem, we present a result in this direction by proving that there does not exist any rank one extreme contraction in $ \mathbb{L}(\mathbb{H}, \ell_{p}^2), $ whenever $ p $ is even and $\mathbb{H}$ is any Hilbert space. Clearly, this is outside the scope of \cite{G,BR} and moreover, this also illustrates how drastically the nature of extreme contractions can change if we chose the domain and the range space different from one another.\\

\begin{theorem}\label{theorem:ell24rank1}
	 Let $ T$ be a rank one operator from $\mathbb{H}$ to $\ell_p^2$, where $p$ is even and $\mathbb{H}$ is any Hilbert space. Let $\|T\|=1$. Then $T$ is not an extreme contraction.
\end{theorem}
\begin{proof}
	We first show that for any $\widetilde{x}\in S_{\mathbb{H}}$ and $(x_1,~y_1) \in S_{\ell_p^2}$, $(\widetilde{x},~(x_1,~y_1))$ is a CPP, if $p$ is even. Let $p=2m$.  Choose $r, ~0 <r<1$.  Let $ \widetilde{u}\in \widetilde{x}^{\perp}\cap S_{\mathbb{H}}$. Then $a\widetilde{x}+b\widetilde{u}\in B(\widetilde{x},~r)\cap S_{\mathbb{H}}$ gives that $a^2+b^2=1$, $|a-1|<r$ and $|b|<r$. Clearly, $(x_1,y_1)\perp_B (-y_1^{p-1},x_1^{p-1})$. Let $k_1=\frac{1}{\|(-y_1^{p-1},x_1^{p-1})\|}$. In order to prove our claim, we have to find $\mu >0$ such that $\|a(x_1,y_1)+b\mu k_1(\mp y_1^{p-1}, \pm x_1^{p-1})\|^p \leq 1$.
	 Now,  $\|a(x_1,y_1)+b\mu k_1(-y_1^{p-1},x_1^{p-1})\|^p\leq 1 $ 
	if and only if $ a^p+\binom{p}{2}(ax_1y_1)^{p-2}(b\mu k_1)^2+\binom{p}{3}(ax_1y_1)^{p-3}(b \mu k_1)^3(x_1^p-y_1^p)+ \ldots + (b\mu)^p \leq  1 = (a^2+b^2)^m .$ This holds if and only if  
	\begin{eqnarray*}
	\binom{p}{2}(ax_1y_1)^{p-2}(b\mu k_1)^2+ & & \binom{p}{3}(ax_1y_1)^{p-3}(b \mu k_1)^3(x_1^p-y_1^p)  +  \ldots  \ldots + (b\mu)^p  \\
	& & \leq \binom{m}{1}a^{p-2}b^2+\binom{m}{2}a^{p-4}b^4+\ldots + b^p.
	\end{eqnarray*}
	 Choose $ 0<\mu \leq 1$ such that
	\begin{eqnarray*}
	\mu \Bigg[ \binom{p}{2}(1+r)^{p-2}|x_1y_1|^{p-2}k_1^2 & + & \binom{p}{3}(1+r)^{p-3}rk_1^3|x_1y_1|^{p-3}|x_1^p-y_1^p|  \\
	& + & \ldots + r^{p-2}\Bigg] \leq \binom{m}{1}(1-r)^{p-2}.
	\end{eqnarray*}
	 Then 
	\begin{eqnarray*}
	 & &\binom{p}{2}(ax_1y_1)^{p-2}(\mu k_1)^2+\binom{p}{3}(ax_1y_1)^{p-3}b( \mu k_1)^3(x_1^p-y_1^p) + \ldots + b^{p-2} \mu^p  \\
	&\leq& \mu  \Bigg[\binom{p}{2}(1+r)^{p-2}|x_1y_1|^{p-2}k_1^2+\binom{p}{3}(1+r)^{p-3}rk_1^3|x_1y_1|^{p-3}|x_1^p-y_1^p| \\
	&  & ~~+ \ldots \ldots + r^{p-2}\Bigg] \\
	 & \leq &    \binom{m}{1}(1-r)^{p-2} \\
	& \leq & \binom{m}{1}a^{p-2}+\binom{m}{2}a^{p-4}b^2+\ldots + b^{p-2}.
	\end{eqnarray*}
	Hence, 
	  \[\|a(x_1,y_1)+b\mu k_1(-y_1^{p-1},x_1^{p-1})\|^p\leq 1 .\]
		Similarly, \[\|a(x_1,y_1)+b\mu k_1(y_1^{p-1},-x_1^{p-1})\|^p\leq 1 .\]
	  Thus $(\widetilde{x},~(x_1,y_1))$ is a CPP.
	 Therefore, if $T\in \mathbb{L}(\mathbb{H},~\ell_p^2)$ be such that $\|T\|=1$ and $rank~T=1$ then by Theorem \ref{theorem:completeHrank1}, we  conclude that  $T$ is not an extreme contraction. 
\end{proof}

Using similar method, we can characterize the class of all rank one extreme contraction in $ \mathbb{L}(\ell_{4}^2, \mathbb{H}) $, where $\mathbb{H}$ is any Hilbert space, in the form of the following theorem: 
 
\begin{theorem}\label{theorem:ell4hrank1}
	Let $\mathbb{H}$ be any Hilbert space and $T$ be a rank one operator in $ \mathbb{L}(\ell_4^2,\mathbb{H}) $ with $\|T\|=1$. Then $T$ is an extreme contraction if and only if $M_T=\{\pm(1,0)\}$ or $M_T=\{\pm(0,1)\}$. In particular, if $\mathbb{H}=\ell_2^n$ then  the matrix representations (with respect to the standard ordered basis) of the class of all rank one  extreme contractions are given by $$
	\begin{bmatrix}
	x_1 & 0\\
	x_2 & 0	\\
	\vdots & \vdots \\
	x_n & 0
	\end{bmatrix}
	,
	\quad
	\begin{bmatrix}
	0 & x_1\\
	0 & x_2\\
	\vdots & \vdots \\
	0 & x_n	
	\end{bmatrix}
	,$$ 
	where, $x_1^2+x_2^2+ \ldots x_n^2=1$.
\end{theorem}

\begin{remark}
	It follows from Theorem \ref{theorem:ell4hrank1} that in $ \mathbb{L}(\ell_4^2,\ell_2^2), $ the matrix representations (with respect to the standard ordered basis) of the class of all rank one extreme contractions in $ \mathbb{L}(\ell_4^2, \ell_2^2) $ are given by $$
	\begin{bmatrix}
	x_1 & 0\\
	y_1 & 0	
	\end{bmatrix}
	,
	\quad
	\begin{bmatrix}
	0 & x_1\\
	0 & y_1	
	\end{bmatrix}
	,$$ 
	where, $x_1^2+y_1^2=1$. We would like to note that this description of rank one extreme contractions in $ \mathbb{L}(\ell_4^2, \ell_2^2) $ also follows from Theorem 2.2 of \cite{BR}. However, Theorem \ref{theorem:ell4hrank1} is applicable for any Hilbert space as the range space, whereas, Theorem 2.2 of \cite{BR} is applicable only in the two-dimensional case. 

Let us further note that the significance of the domain space and the range space in the study of extreme contractions between Banach spaces is further illustrated by Theorem \ref{theorem:ell24rank1} and Theorem \ref{theorem:ell4hrank1}. In $\mathbb{L}(\mathbb{H},\ell_4^2) $, there does not exist any rank one extreme contraction. However, in $\mathbb{L}(\ell_4^2,\mathbb{H}) $, indeed there are rank one extreme contractions. 
\end{remark}

The following  example shows that the condition $rank~T=1$  in Theorem \ref{theorem:rank1} is not necessary for $T$ to be such that $ T $ is not an extreme contraction.  
\begin{example}
	Define $T:{\ell_p}^2 \longrightarrow {\ell_p}^2$ by 
	\[T(1,~0)=(1,~0),~T(0,~1)=\frac{1}{2}(0,~1).\]
	Then it is easy to check that $M_T=\{\pm(1,~0)\}$. Again $((1,~0),~(1,~0))$ is a CPP. From \cite{G}, it is easy to see that $T$ is not an extreme contraction.
\end{example}
The above example calls for our attention towards the problem of characterizing rank two extreme contractions on a two-dimensional strictly convex and smooth Banach space. In this context, it is worthwhile to observe that if a linear operator $ T $, defined between two-dimensional strictly convex Banach spaces, attains norm at two linearly independent unit vectors then $ T $ is surely an extreme contraction. Therefore, it is sufficient to restrict our attention to rank two linear operators between two-dimensional strictly convex and smooth Banach spaces, which attain norm at only one pair of points.  As we will see in the following theorem, in order to obtain a desired characterization in this case, the notion of $ \mu- $CPP plays a vital role. We consider $T\in \mathbb{L(\mathbb{X}, \mathbb{Y})}$, where, $\mathbb{X},~\mathbb{Y}$ are two-dimensional smooth Banach spaces and $\mathbb{Y}$ is strictly convex. Let $\|T\|=1,~rank~T=2,~M_T=\{\pm x\}$. Let $y\in x^{\perp}\cap S_{\mathbb{X}}$. Then clearly $0<\|Ty\|<1$. As $\mathbb{X}$ and $\mathbb{Y}$ are both smooth so $Tx\perp_B Ty$. Let $w=\frac{Ty}{\|Ty\|}$. \\
In the following theorem we prove that $T$ is not an extreme contraction if and only if $(x,~Tx)$ is a $\mu$-CPP with respect to $(y,~w)$ for some $\mu >\|Ty\|$.
\begin{theorem}
Let $\mathbb{X},~\mathbb{Y}$ be two-dimensional smooth Banach spaces and in addition, $\mathbb{Y}$ be strictly convex. Let $ T$ be a rank two operator from $\mathbb{X}$ to $\mathbb{Y}$ with $\|T\|=1$ and $M_T=\{\pm x\}$. Then $T$ is not an extreme contraction if and only if  $(x,~ Tx)$ is a $\mu-$CPP with respect to the pair $(y,~w)$ for some $\mu > \|Ty\|$.
\end{theorem} 
\begin{proof}	
We first prove the necessary part of the theorem.\\
 Let $\|Ty\|=k$. Suppose that  $T$ is not an extreme contraction. Then there exists $T_1,~T_2 \in \mathbb{L}(\mathbb{X}, \mathbb{Y})$ such that $T=\frac{1}{2}(T_1+T_2)$, $T \neq T_1,~T\neq T_2,~\|T_1\|=\|T_2\|=1$. Clearly $Tx=  \frac{1}{2}(T_1 x+T_2 x)$. Since $Tx$ is an extreme point of $B_{\mathbb{Y}}$, we have, $Tx=T_1 x= T_2 x$, i.e., $x \in M_{T_1} \cap M_{T_2}$. Therefore, again applying Theorem 2.2 of \cite{S}, we have, $x\perp_B y \Longrightarrow T_1 x \perp_B T_1 y \Longrightarrow Tx \perp_B T_1y$. Therefore, $T_1 y=(k+ \delta)w$ and $T_2 y= (k- \delta)w$ for some $\delta \in \mathbb{R} \setminus \{0\}$. Let $ax+by \in S_{\mathbb{X}}$. Then $\|T_1(ax+by)\|\leq 1 \Longrightarrow \|aTx+b(k+\delta)w\|\leq 1$ and $\|T_2(ax+by)\|\leq 1 \Longrightarrow \|aTx+b(k-\delta)w\|\leq 1$. Therefore, $\|aTx+b(k+|\delta|) Ty\|\leq 1$. Thus, $(x,~ Tx)$ is a $\mu-$CPP with respect to the pair $(y,~w)$ for some $\mu >\|Ty\|$. This completes the proof of the necessary part of the theorem.\\
For the sufficient part, assume that $(x,~ Tx)$ is a $\mu-$CPP with respect to the pair $(y,~w)$ for some $\mu > \|Ty\|=k$. Let $\mu=k+\epsilon$ for some $\epsilon >0$. Therefore, there exists $r>0$ such that $ax+by \in B(x,~r) \cap S_{\mathbb{X}}$ implies that $\|aTx+ b(k+\epsilon)w\|\leq 1$. Since $M_T=\{\pm x\}$, there exists $\delta >0$ such that $\sup \{\|Th\|:~h \in B(x,~r)^c \cap S_{\mathbb{X}}\}=1- \delta$. Choose $0< c < min\{k, \frac{\delta}{2},\epsilon\}$. Define linear operators $T_1,~T_2$ from $\mathbb{X}$ to $\mathbb{Y}$ as follows:
\begin{align*}
T_1x & = Tx &    T_2 x & = Tx\\
T_1y & = (k+c)w  &   T_2y & = (k-c)w\\
\end{align*}
Clearly, $T=\frac{1}{2}(T_1+T_2),~ T \neq T_1,~T \neq T_2$.
Let $ax+by \in B(x,~r) \cap S_{\mathbb{X}} $. Then since $\|aTx+b(k+\epsilon)w\|\leq 1$, using $0<k-c < k+c\leq k+\epsilon$ and Proposition \ref{prop:convexity}, we obtain, $\|T_1(ax+by)\|=\|aTx+b(k+c)w\|\leq 1$ and $\|T_2(ax+by)\|=\|aTx+b(k-c)w\|\leq 1$. Now, let $ax+by \in B(x,~r)^c \cap S_{\mathbb{X}}$. Since $x\perp_B y$ and $\|ax+by\|=1$, $|b|\leq \|ax+by\|+\|ax\|\leq 2$. Then $\|T_1(ax+by)\|=\|aTx+b(k+c)w\|=\|T(ax+by)+b c w\| \leq \|T(ax+by)\|+|bc| \leq 1-\delta + 2c <1-\delta +\delta =1$. Therefore, $\|T_1\|= 1$. Similarly it can be shown that $\|T_2\|=1$. Hence, $T$ is not extreme contraction.
\end{proof}

Now we are in a position to completely characterize extreme contractions between two-dimensional strictly convex and smooth Banach spaces. We would like to remark that to the best of our knowledge, such a general characterization of extreme contractions is being presented for the very first time. Indeed, we have the following theorem:

\begin{theorem}
	Let $\mathbb{X},~\mathbb{Y}$ be two-dimensional smooth Banach spaces and in addition, $\mathbb{Y}$ be strictly convex. Let $T\in \mathbb{L}(\mathbb{X}, \mathbb{Y})$ be a norm one linear operator. Then $ T $ is an extreme contraction if and only if exactly one of the following three conditions holds true:\\
	(i):\\
	 (a) Rank $ T $ is one,\\
	    (b) $ M_T = \{ \pm x\}  $ for some $ x \in S_{\mathbb{X}}, $ \\
	    (c) $ (x,~Tx) $ is not a CPP.\\
	    
\noindent (ii):\\
 (a) Rank $ T $ is two,\\
	     (b) $ M_T = \{ \pm x\}  $ for some $ x \in S_{\mathbb{X}}, $\\
	     (c) $ Ty = k w $ for some $ y \in x^{\perp} \cap S_{\mathbb{X}}, $  $ w \in (Tx)^{\perp} \cap S_{\mathbb{X}} $ and $ 0 < k <1. $ \\
	     (d)  $(x,~ Tx)$ is not a $\mu-$CPP with respect to the pair $(y,~w)$ for any $\mu >k$.\\
	     
\noindent (iii):\\	
           (a) Rank $ T $ is two,\\
           (b) $ T $ attains norm at at least two pairs of linearly independent unit vectors $ \pm x,~ \pm y. $

\end{theorem}

So far, our analysis has been mostly confined to two-dimensional (strictly convex and smooth) Banach spaces. A more general result, without any restriction on the dimension of the space, is given in the form of the following theorem:

\begin{theorem}
	Let $\mathbb{X}$ be a smooth, reflexive Banach space with Kadets-Klee property and $\mathbb{Y}$ be a smooth Banach space. Let $T\in \mathbb{K}(\mathbb{X},\mathbb{Y})$ be such that $\|T\|=1$ and $ T $ is a smooth operator. Then $M_T= \{\pm x\}, $ for some $  x \in S_{\mathbb{X}}. $ Moreover, if we additionally assume that $(x,~ Tx)$ is a CPP with CPP constants $(r,~\mu)$ and $0< \|T\|_{H} < \mu$, where $H$ is the unique hyperspace  such that  $x \perp_B H$, then $T$ is not an extreme contraction.
\end{theorem}
 \begin{proof}
Let us first observe that, since $T\in \mathbb{K}(\mathbb{X}, \mathbb{Y})$ is smooth, it follows from Theorem 4.2 of \cite{PSG} that $ M_T= \{\pm x\}, $ for some $ x \in S_{\mathbb{X}}. $ Since $\|T\|_{H}>0$, we can choose $y \in H \cap S_{\mathbb{X}}$ such that $Ty \neq 0$. Then in the subspace $ H, $ there exists a unique hyperspace $ H_0$ such that $y \perp_B H_0$. Hence, every $z \in \mathbb{X}$ can be uniquely written as $ax+by+ h_0$, where $a,~b \in \mathbb{R},~h_0 \in H_0 $. 
For each $n \in \mathbb{N},$ define $T_n,~ S_n $ on $\mathbb{X}$ as follows:
	\[T_n(z)=a Tx+(1+\frac{1}{n})(bTy+Th_0),\]
	\[ S_n(z)=a Tx+(1-\frac{1}{n})(bTy + Th_0).\]
 		Using the sequential criteria of a compact operator, it is easy to see that for each $n \in \mathbb{N},$ $T_n,~S_n \in \mathbb{K}(\mathbb{X}, \mathbb{Y})$.
 	 Clearly, $T= \frac{1}{2}T_n + \frac{1}{2}S_n$ and $T\neq T_n$, $T\neq S_n$  for all $n\in \mathbb{N}$. It is easy to verify that $\|S_n\|\leq 1$ for all $n \in \mathbb{N}$. If $\|T_n \| \leq 1$ for any $n$ then $ T$ is not an extreme contraction and we are done. So let us assume that $\|T_n\|>1$ for each $n \in \mathbb{N}$. Each $T_n$, being a compact operator on a reflexive Banach space, attains its norm. Let $y_n \in M_{T_n}$ for each $n \in \mathbb{N}$. \\
 	Now, since $0<\|T\|_H < \mu $, there exists $n_0\in \mathbb{N}$ such that $\frac{\mu}{\|T\|_H}>1+ \frac{1}{n}$ for all $n \geq n_0$.  	 Let $z= (ax+by+h_0) \in B(x,~ r)\cap S_{\mathbb{X}}$. We show that $\|T_n z\|\leq 1$  for all $n \geq n_0$. If $ by + h_0 = 0$ then clearly $\|T_nz \| \leq 1$ for each $ n \in \mathbb{N}.$ Let $by+h_0 \neq 0$. Then $z=ax+ch$, where $c=\|by+h_0\|$ and $h=\frac{1}{c}(by+h_0)$. Clearly,  $|a|\leq \|ax+ ch\|=\|z\|=1$ and  $|b|\leq \|by+h_0\|=\|ch\|=c$. First assume $Th\neq 0$. Then since $z=a x+ c h \in B(x,~ r) \cap S_{\mathbb{X}}$ and $(x,~ Tx)$ is a CPP with CPP constant $(r,~\mu)$, we have, $ \|a Tx \pm \frac{c \mu}{\|Th\|}Th \|\leq 1$. Again, $\frac{\mu}{\|Th\|}\geq \frac{\mu}{\|T\|_H}>1+ \frac{1}{n}$ for all $n \geq n_0$ and so using Proposition \ref{prop:convexity} we get, 	 
 	  \[\|T_n z\|  =    \|aTx + (1+\frac{1}{n})\{bTy+Th_0\}\| = \|a Tx+ c (1+\frac{1}{n})Th\|  \leq  1\]
Next, let us assume $Th=0$. Then $\|T_nz\|=\|aTx+(1+\frac{1}{n})(bTy+Th_0)\|=\|aTx+(1+\frac{1}{n})cTh\|=\|aTx\|\leq 1$.
 Since $y_n \in M_{T_n} $ and $ \|T_n \| >1, $ we have,  
  $y_n \notin B(x,~ r)$ for all $n \geq n_0$. Similarly, it can be shown that  $y_n \notin B(-x,~r)$ for all $n \geq n_0$. \\
 Since $\mathbb{X}$ is reflexive, $B_{\mathbb{X}}$ is weakly compact and so $\{y_n\}$ has a weakly convergent subsequence in $B_{\mathbb{X}}$. Without loss of generality, let us assume that $\{y_n\}$ converges weakly to $y_0 \in B_{\mathbb{X}}$. Let $z=(a x + b y + h_0)\in S_{\mathbb{X}}$. Then $|a|\leq \|(a x + by+h_0)\|=1$ and $ \|( b y + h_0)\|= \|z-a x\|\leq 1+ |a| \leq 2$. Now, $\|(T_n-T)z\|=\|(T_n-T)(ax+by+h_0)\|=\|\frac{1}{n}(bTy+Th_0)\|\leq \frac{1}{n}\|by+h_0\| \leq \frac{2}{n}$. This implies that $\|(T_n-T)z\|\longrightarrow 0$ as $n \longrightarrow \infty$. Hence, $T_n \longrightarrow T$. Since $T$ is a compact operator and $\{y_n\}$ converges weakly to $y_0$, $Ty_n$ converges strongly to $Ty_0$. Therefore, $\|T_n y_n- Ty_0\|\leq \|T_n y_n-Ty_n\|+ \|Ty_n-Ty_0\|\leq \|T_n-T\|+\|T y_n-Ty_0\|\longrightarrow 0$  as $n \longrightarrow \infty$. Hence, $T_ny_n \longrightarrow Ty_0 \Longrightarrow \|T_n y_n\|\longrightarrow \|Ty_0\| \Longrightarrow \|Ty_0\|\geq 1$, since $\|T_n y_n\|>1$ for all $n \in \mathbb{N}$. Since $y_0 \in B_{\mathbb{X}}$ and $\|T\|=1, $ we have, $\|Ty_0\|=1$. Therefore, $y_0 = \pm x $. Again, $y_n \rightharpoonup y_0$ and $\|y_n\|\longrightarrow \|y_0\|$. Hence, $y_n \longrightarrow y_0$. This implies that $y_0 \notin B(\pm x,~ r)$, a contradiction.  Therefore, $T$ is not an extreme contraction. This establishes the theorem.
 \end{proof}

We would also like to  illustrate that extreme contractions and this newly introduced notion of CPP play an important role in determining the geometry of the underlying space. First we deal with strict convexity. Indeed, our next two theorems separately give a characterization of strict convexity (in the two-dimensional case) and a sufficient condition for strict convexity (without any restriction on the dimension of the underlying space), using the notions of extreme contractions and CPP.

\begin{theorem}
	Let $\mathbb{X}$ be a two dimensional Banach space. Suppose for any $T \in \mathbb{L}(\mathbb{X})$, $T$ attains norm at least at two linearly independent vectors implies that $T$ is an extreme contraction. Then $\mathbb{X}$ is strictly convex. 
\end{theorem}
\begin{proof}
	If possible, suppose that $\mathbb{X}$ is not strictly convex. Then there exists $x,~y \in S_{\mathbb{X}}$ and $\lambda_0 >0$ such that the line segment joining $x+ \lambda_0 y$ and $x- \lambda_0 y$ is a subset of $S_{\mathbb{X}}$. This clearly implies that $x\perp_B y$.	
	 Define linear operators $T, ~T_1,~T_2$ on $\mathbb{X}$ as follows:
	\begin{align*}
	Tx & = x &    T_1x & = x+ \frac{1}{2}\lambda_0 y & T_2 x & = x- \frac{1}{2}\lambda_0 y \\
	Ty & = 0  &   T_1y & = 0                         & T_2 y & = 0                        
	\end{align*}
	Then it is easy to check that $T$ attains norm at two linearly independent vectors $x+ \lambda_0 y$ and $x- \lambda_0 y$. Also it is easy to verify that $\|T\|=\|T_1\|=\|T_2\|=1$, for if  $z=ax+by \in S_{\mathbb{X}}$ then $\|T_1 z\|=\|a(x+\frac{1}{2}\lambda_0 y)\|=|a|\leq \|ax+by\|=1 $ and $ \|T_1 x\|=1$. Thus $T=\frac{1}{2}(T_1+T_2)$ and $T \neq T_1,~T\neq T_2$ shows that  $T$ is not an extreme contraction, contradicting the hypothesis. Therefore, $\mathbb{X}$ must be strictly convex.
\end{proof}

\begin{theorem}\label{theorem:strictlyconvex}
	Let $\mathbb{X}$ be a Banach space such that for any $x \in S_{\mathbb{X}}$, $(x,~x)$ is a CPP. Then $\mathbb{X}$ is strictly convex.
\end{theorem}
\begin{proof}
	If possible, suppose that $\mathbb{X}$ is not strictly convex. Then there exists $y,~z \in S_{\mathbb{X}}$ and $\lambda_0 >0$ such that the line segment joining $y+ \lambda_0 z$ and $y- \lambda_0 z$ is a subset of $S_{\mathbb{X}}$. Let $S= \{\lambda>0:\|y+ \lambda z\|>1\}$. Then $S$ is nonempty, for, $\|y+ \lambda z\|\geq |\lambda|-1> 1$ for $|\lambda|>2$. Let $k= \inf S$. Let $x= y+ k z$. If $\|y+ k z\|>1$ then clearly there exists $\delta >0$ such that $\|y+ \lambda z\|>1$ for all $\lambda \in (k- \delta,~k+ \delta)$, which contradicts the fact that $k= \inf S$. So $\|x\|=1$. Clearly, $x \perp_B z$ and $\|x- \lambda z\|=1$ for all $\lambda \in [0,k]$. We claim that $(x,~x)$ is not a CPP which leads to a contradiction. Let $r>0$. Then we can find $\lambda \in [0,k]$ such that $x- \lambda z \in B(x,r)\cap S_{\mathbb{X}}$ but $\|x+ \mu \lambda z\|>1$ for any $\mu >0$. This completes the proof of the theorem.
\end{proof}

As the final result of this paper, we obtain a complete characterization of Hilbert spaces among Banach spaces in terms of CPP. 

\begin{theorem}
	A Banach space $\mathbb{X}$ is a Hilbert space if and only if for any $x \in S_{\mathbb{X}}$, $(x,~x)$ is a CPP with CPP constant $(r,~1)$ for any $r>0$. 
\end{theorem}
\begin{proof} The necessary part follows easily from Proposition ~\ref{proposition:properties}(ii). For the sufficient part, we show that if $x,~y\in S_{\mathbb{X}}$ and $x\perp_B y$ then $x\perp_I y$. 
	Since for any $x \in S_{\mathbb{X}}$, $(x,~x)$ is a CPP, from Theorem \ref{theorem:strictlyconvex}, we have, $\mathbb{X}$ is strictly convex. Let $x,~y \in S_{\mathbb{X}}$ and $ x\perp_B y$. Then $\|x+y\|>1$. Let $c=\frac{1}{\|x+y\|}$. Then $0<c< 1$ and $c(x+y)\in S_{\mathbb{X}}$. Clearly, $c(x+y)\in B(x,~1)$ and so  $\|cx-cy\|\leq 1$, since $(x,~x)$ is a CPP. We claim that $\|cx-cy\|=1$. If possible, suppose that $\|cx-cy\|< 1$. Define a function $f:\mathbb{R}\longrightarrow \mathbb{R}$ by 
		\[f(a)=\|cx-ay\|.\]
	Clearly, $f$ is continuous on $\mathbb{R}$ and $f(c)<1$. Also, $f(c+2)=\|cx-(c+2)y\|\geq c+2-c=2 >1$. Therefore, by continuity of $ f, $ there exists some	 $\lambda \in (c,~c+2) $ such that $f(\lambda)=1$, i.e., $\|cx-\lambda y\|=1$. Since $(x,~ x)$ is a CPP, $\|cx+\lambda y\|\leq 1$.Then
	\begin{eqnarray*}
	   cx+cy &=& (1-\frac{c}{\lambda})cx + \frac{c}{\lambda}(cx+\lambda y)\\
	   \Longrightarrow \| cx+cy\| &\leq & (1-\frac{c}{\lambda})\|cx\| + \frac{c}{\lambda}\|(cx+\lambda y)\|\\
	    \Longrightarrow \| cx+cy\| &< & (1-\frac{c}{\lambda}) + \frac{c}{\lambda}\\
	    \Longrightarrow \| cx+cy\| &< & 1,~ a ~ contradiction.
	\end{eqnarray*} 
 Therefore, $\|cx-cy\|=1$ and so $\|x+y\|=\|x-y\|$.  Hence, $x\perp_B y \Longrightarrow x\perp_I y$. Finally, let us observe that from Theorem 5.1 of \cite{AMW}, it now follows that $\mathbb{X}$ is a Hilbert space. This completes the proof of the theorem. 
\end{proof}

\begin{acknowledgement}
Dr. Debmalya Sain feels elated to acknowledge the tremendous positive contribution of the following three persons: Professor Gadadhar Misra, for introducing the author to the study of extreme contractions; Professor Vladimir Kadets, for the constant help, encouragement and inspiration; Mrs. Krystina DeLeon, for her beautiful friendship.
\end{acknowledgement}


\begin{thebibliography}{99}

\bibitem{AMW} Alonso, J., Martini, H., Wu, S.,  \textit{On Birkhoff orthogonality and isosceles orthogonality in normed linear spaces}, Aequationes Math., \textbf{83} (2012) 153-189.

\bibitem{BR} Bandyopadhyay, P. and Roy, A. K., \textit{Extreme contractions in $\mathbb{L}(\ell_p^2,\ell_q^2)$ and the Mazur intersection property in $\ell_p^2 \otimes_{\pi} \ell_q^2$}, Real Anal. Exchange, \textbf{20} (1994/95) 681-698.

\bibitem{B} Birkhoff, G.,  \textit{Orthogonality in linear metric spaces}, Duke Math. J., \textbf{1} (1935) 169-172.


\bibitem{G}  Grz\c a\' slewicz, R.,  \textit{Extreme operators on 2-dimensional lp-spaces}, Colloquium Mathematicum, \textbf{44} (1981) 309--315. 

\bibitem{Ga} Grz\c a\' slewicz, R.,  \textit{Extreme operators on real Hilbert spaces},  Math. Ann., \textbf{261} (1982) 463-466. 



\bibitem{J}James, R.C., \textit{Orthogonality and linear functionals in normed linear spaces}, Transactions of the American Mathematical Society, \textbf{61} (1947b) 265-292.


\bibitem{I} Iwanik, A.,  \textit{Extreme contractions of certain function spaces}, Colloquium Math., \textbf{40} (1978) 147-153. 


\bibitem{K} Kadison, R. V.,  \textit{Isometries of operator algebras}, Ann. Math., \textbf{54} (1951), 325-338. 

\bibitem{Ki} Kim, C. W.,  \textit{Extreme contraction operators on $l_\infty$}, Math. Zeitschrift, \textbf{151} (1976) 101-110. 


	\bibitem{PSG}  Paul, K.,   Sain, D.  and Ghosh, P.,
\textit{ Birkhoff-James orthogonality and smoothness of bounded linear operators},
Linear Algebra and its Applications,
\textbf{506} (2016) 551-563.

\bibitem{S} Sain, D., \textit{On the norm attainment set of a bounded linear operator}, J.Math.Anal. Appl., \textbf{457} (2018) 67-76.

\bibitem{Sh} Sharir, M., \textit{Characterization and properties
of extreme operators into C(Y)}, Israel Journal of Mathematics, \textbf{12},  (1972) 174-183.
\end{thebibliography}
\end{document}